\documentclass[12pt]{amsart}

\usepackage[a4paper,centering,margin=2.5cm]{geometry}
\usepackage{cite}
\usepackage{color,hyperref}
\definecolor{darkblue}{rgb}{0.0,0.0,0.3}
\hypersetup{colorlinks,breaklinks,
            linkcolor=darkblue,urlcolor=darkblue,
            anchorcolor=darkblue,citecolor=darkblue}

\usepackage{amssymb}
\usepackage{graphicx,psfrag}
\def\boxit{$\sqcap\kern-8pt\sqcup$}

\title{How many circuits determine an oriented matroid?}

\dedicatory{This paper is dedicated to the memory of Michel Las Vergnas}

\author{Kolja Knauer}
\address[Kolja Knauer]{Laboratoire d'Informatique Fondamentale, Aix-Marseille Universit\'e and CNRS, Facult\'e des Sciences de Luminy, F-13288 Marseille Cedex 09, France}
\email{kolja.knauer@lif.univ-mrs.fr}

\author{Luis Pedro Montejano}
\address[Luis Pedro Montejano]{Universit\'e de Montpellier, Institut Montpelli\'erain Alexander Grothendieck, Case Courrier 051, Place Eug\`ene Bataillon, 34095 Montpellier Cedex 05, France}
\email{lpmontejano@gmail.com}

\author{Jorge Luis Ram\'irez Alfons\'in}
\address[Jorge Luis Ram\'irez Alfons\'in]{Universit\'e de Montpellier, Institut Montpelli\'erain Alexander Grothendieck, Case Courrier 051, Place Eug\`ene Bataillon, 34095 Montpellier Cedex 05, France}
\email{jramirez@um2.fr}

\thanks{The authors were supported by grant ANR-10-BLAN 0207 and and PICS06316}
\thanks{The first author was also supported by ANR grant EGOS ANR-12-JS02-002-01 and PEPS grant EROS}
\thanks{The second author was also supported by the grant LAISLA}

\date{\today}

\theoremstyle{plain}
\newtheorem{theorem}[subsection]{Theorem}
\newtheorem{lemma}[subsection]{Lemma}
\newtheorem{proposition}[subsection]{Proposition}
\newtheorem{corollary}[subsection]{Corollary}
\newtheorem{observation}[subsection]{Observation}

\newtheorem{conjecture}{Conjecture}
\newtheorem{example}{Example}

\newcommand{\C}{\mathrm{C}}
\newcommand{\Cc}{\mathrm{C}}

\newcommand{\CC}{\mathrm{CC}}
\newcommand{\WC}{\mathrm{WC}}
\newcommand{\ccc}{\mathrm{cc}}
\newcommand{\cc}{\mathrm{c}}
\newcommand{\cbc}{\mathrm{cbc}}
\newcommand{\bc}{\mathrm{bc}}

\newcommand{\crc}{\mathrm{circ}}

\newcommand{\M}{\mathcal{M}}
\newcommand{\Mm}{\mathrm{M}}
\renewcommand{\S}{\mathcal{S}}
\newcommand{\Ss}{\mathrm{S}}

\begin{document}

\begin{abstract} 
Las Vergnas \& Hamidoune studied the number of circuits needed to determine an 
oriented matroid. In this paper we investigate this problem and some new variants, 
as well as their interpretation in particular classes of matroids.
We present general upper and lower bounds in the setting of general connected orientable matroids,
leading to the study of subgraphs of the base graph and the intersection graph of circuits. 

We then consider the problem for uniform matroids which is closely related to 
the notion of \emph{(connected) covering numbers} in Design Theory. Finally, we also devote special
attention to regular matroids as well as some graphic and cographic matroids leading 
in particular to the topics of (connected) bond and cycle covers in Graph Theory.
%
%
\end{abstract}

\maketitle

\section{Introduction}

For the general background on matroid and oriented matroid theory we refer the reader to \cite{Ox-99} and \cite{Bjo-99}, respectively.
An (oriented) matroid is a finite ground set together with a (usually large) set of (oriented) circuits satisfying certain axioms. {But, how many of these circuits are actually needed to fully describe a given (oriented) matroid ?} 

In~\cite[page 721]{Leh-64} Lehman shows, that the set $\Ss_e$ of circuits of a connected matroid $\Mm$ containing a fixed element $e$, distinguishes $\Mm$ from all other matroids on the same ground set, that is, if a matroid $\Mm'$ on the same ground contains all circuits from $\Ss_e$ then  $\Mm$ and $\Mm'$ are the same. 

\smallskip

Las Vergnas and Hamidoune~\cite{Ham-86} extend Lehman's result to  {an oriented version}. They prove that a connected oriented matroid $\M$ is uniquely determined by the collection of signed circuits $\S_e$ containing a given element $e$, i.e., if an oriented matroid $\M'$ on the same ground set as $\M$ contains all circuits from $\S_e$ then $\M$ and $\M'$ are the same. 

%

In view of Las Vergnas and Hamidoune's result, one may ask the following natural question:

\begin{quote}
{\em How many circuits are needed to determine a connected oriented matroid?}
\end{quote}

\subsection{Scope/general interest} It turns out that the above question can be interpreted in different ways. In this paper,  we will investigate the number of circuits needed to determine an oriented matroid among all oriented matroids with the same underlying matroid.  Let us introduce some notation in order to explain this more precisely. Generally, we represent matroids and oriented matroids as pairs of a ground set and a set of (signed) circuits. Throughout this paper we use calligraphic letters for sets of signed circuits and oriented matroids and non-italic roman letters for the non-oriented case. We say that two matroids $\mathrm{M}_1=(E_1,\mathrm{C}_1)$ and $\mathrm{M}_2=(E_2,\mathrm{C}_2)$ are the \emph{same}, i.e., $\mathrm{M}_1=\mathrm{M}_2$ if and only if $E_1=E_2$ and $\mathrm{C}_1=\mathrm{C}_2$.
Note that the equality is more restrictive than \emph{isomorphism} even when restricted to the same ground set, where the latter means that there is a permutation of the ground set which preserves circuits. This is illustrated in the following

\begin{example}\label{exmpl:isomorphic}
 Let $\mathrm{M}(G_1),\mathrm{M}(G_2)$ and $\mathrm{M}(G_3)$ be the graphic matroids associated to the graphs  $G_1,G_2$ and $G_3$ given in Figure~\ref{fig:ex1}. 
 
 \begin{figure}[htb] 
 \includegraphics[width=.5\textwidth]{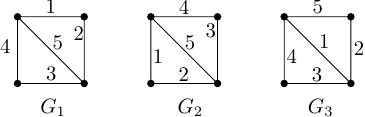}
 \caption{} \label{fig:ex1}
\end{figure}

We clearly have that $\mathrm{M}(G_1)=\mathrm{M}(G_2)$ since $$\Cc(\mathrm{M}(G_1))=\{\{1,2,5\},\{3,4,5\},\{1,2,3,4\}\}=\Cc(\mathrm{M}(G_2)).$$ 
However, although $\mathrm{M}(G_1)$ and $\mathrm{M}(G_3)$ are isomorphic by taking the permutation
$\pi(1)=5,\pi(2)=2,\pi(3)=3,\pi(4)=4$ and $\pi(5)=1$ we have that
$\mathrm{M}(G_1)\neq \mathrm{M}(G_3)$ since $$\Cc(\mathrm{M}(G_1))=\{\{1,2,5\},\{3,4,5\},\{1,2,3,4\}\}\neq \{\{1,2,5\},\{1,3,4\},\{2,3,4,5\}\}=\Cc(\mathrm{M}(G_3)).$$

\end{example}

Similarly as for the non-oriented case, we say that two oriented matroids $\mathcal{M}_1=(E_1,\mathcal{C}_1)$ and $\mathcal{M}_2=(E_2,\mathcal{C}_2)$ are the {\em same}, i.e., $\mathcal{M}_1=\mathcal{M}_2$ if $E_1=E_2$ and $\mathcal{C}_1=\mathcal{C}_2$. For a signed set $\mathcal{X}$ we denote by $\underline{\mathcal{X}}$ its \emph{underlying unsigned set}. We extend this notation to sets of signed sets and furthermore denote by $\underline{\mathcal{M}}$ the \emph{underlying matroid} of the oriented matroid $\mathcal{M}$. In this case we say that $\mathcal{M}$ is an \emph{orientation} of $\underline{\mathcal{M}}$. For a subset of circuits $\mathrm{S}$ of a matroid $\mathrm{M}$ and an orientation $\mathcal{M}$ of $\mathrm{M}$ we denote by $\mathcal{S}_{\mathcal{M}}$ the (maximal) set of signed circuits of $\mathcal{M}$ such that $\underline{\mathcal{S}_{\mathcal{M}}}=\Ss$. We call $\mathcal{S}_{\mathcal{M}}$ the \emph{orientation of $\Ss$ corresponding to $\mathcal{M}$}.
\medskip

We say that a set $\S$ of signed circuits of $\M$ \emph{determines $\M$} if an orientation $\M'$ of $\Mm$ contains the set of signed circuits $\S$  if and only if $\M=\M'$. We say that a set of circuits $\Ss$ of $\Mm$ \emph{determines all orientations of $\Mm$} if for every orientation $\M$ of $\Mm$ the corresponding orientation $\S_{\M}$ of $\Ss$ determines $\M$.  
\smallskip

The above mentioned result of Las Vergnas and Hamidoune  can be restated as follows 

\begin{eqnarray}\label{LVH1}
 \text{the set $\Ss_e$ determines all orientations of $\Mm$.}
\end{eqnarray}
 
{In this spirit,  we define three different quantities for a  connected orientable matroid $\Mm$ in order to investigate the number of  circuits needed to determine it.
\medskip

{Let {\boldmath $s(\Mm)$} be the minimum size of a set $\Ss$ of circuits of $\Mm$ determining all orientations of $\Mm$. }
\medskip

We notice that $s(\Mm)$ requires a {\em fixed} set of (non-oriented) circuits $\Ss$ that 
will be used to determine any orientations of $\Mm$ ($\Ss$ is chosen independently of the orientation of $\Mm$). One may naturally ask whether the size of such a set of circuits can be improved for each fixed orientation.
\medskip

{Let {\boldmath $\widetilde{s}(\mathrm{M})$} be the smallest positive integer $k$ such that in any orientation $\M$ of $M$ there is a set $\S$ of signed circuits of $\M$ of size $k$ that determines $\M$.}
\medskip

{Finally, in this context, we introduce a third variant.}
\medskip

{Let {\boldmath $\overline{s}(\Mm)$} be the smallest $k$ such that \emph{any} set $\Ss$ of circuits of $\Mm$ of size $k$ determines all orientations of $\Mm$.}
\smallskip

The parameter $s(\Mm)$ has been investigated in~\cite{For-98,For-02,Cha-13} for uniform oriented matroids while $\widetilde{s}(\Mm)$ has already been studied in~\cite{daS-07} in connection with a problem about realizability of rank $3$ matroids.  As far as we are aware, {$\overline{s}$ has not been considered before. By the result of Las Vergnas and Hamidoune and simply by definition, respectively, we get

\begin{observation}\label{obs:three} Let $\Mm$ be a connected orientable matroid and $e$ an element. Then,  $$s(\Mm)\le |\Ss_e| \text{ and } \widetilde{s}(\Mm)\leq s(\Mm)\leq\overline{s}(\Mm).$$
\end{observation}

\subsection{Motivations and connections} The quantities  $s(\Mm), \widetilde{s}(\Mm)$ and $\overline{s}(\Mm)$ are natural invariants to be investigated. We expect that these quantities will provide useful and interesting insights in the following appealing and challenging subjects. We leave this for further future work.
\medskip

$\bullet$ The quantities  $s(\Mm), \widetilde{s}(\Mm)$ and $\overline{s}(\Mm)$ may provide efficient ways for encoding oriented matroids by giving only a partial list of circuits. Counting the number of orientations of a given matroid and storing orientations of a matroid is non-trivial. The above quantities may help to simplify this process (and thus bounds on the values of $s(\Mm), \widetilde{s}(\Mm)$ and $\overline{s}(\Mm)$ would be worthwhile).  

\medskip

$\bullet$ In \cite{For-98} the relation between $s(\Mm)$ and both special {\em coverings} in Block Design Theory as well as Tur\'an systems (see Section \ref{sec:uniform}) is put forward and is attractive to study on this context. 

\medskip

$\bullet$ We shall see that some of our results do not rely on the Topological Representation Theorem for oriented matroids but only
on {\em invertible bases}. It might be of interest to investigate whether such results also hold in the wider context of {\em matroids with coefficients} \cite{Dress-91} or maybe for {\em complex matroids} \cite{And12}.

\medskip

$\bullet$ The reorientation classes of a matroid are an important concept in oriented matroid theory. They have a natural graphic (resp. geometric) interpretation for graphic (resp. representable) oriented matroids. In \cite{Gel-95} a characterization of reorientation classes of an oriented matroid by using {\em projective orientations} of $\Mm$ in terms of combinatorics of its circuits and cocircuits is given. This is of interest and significance in the study of stratifications of combinatorial Grassmannians. Therefore, it is attractive to understand the space of oriented matroids over a given matroid (and the quantities $s(\Mm)$ and $\widetilde {s}(\Mm)$). 

\medskip

$\bullet$ An interesting class of oriented matroids is that of those having exactly one reorientation class.  Indeed, it is known ~\cite{Bla-78} that any \emph{regular} oriented matroid has exactly one reorientation class.  In~\cite{Rou-88} it is proved that regular matroids are characterized as those oriented matroids such that all restrictions have only one reorientation class. The following long standing conjecture is due to Las Vergnas ~\cite{Las-89}

\begin{conjecture}
The matroids $C_d$ obtained from the affine $d$-dimensional hypercube has a single reorientation class for all $d$.
\end{conjecture}

$C_d$ is a subclass of cubic matroids~\cite{daS-08}.
It is known \cite{Bok-96} that the conjecture is true when $d\leq 7$. 
\medskip

Investigations on $s(\Mm)$ and $\widetilde {s}(\Mm)$ when $\Mm$ has one reorientation class therefore deserve particular attention. As noticed in Subsection \ref{subsec:reorie} 
and at the beginning of Section~\ref{sec:regular} the quantities $s(\Mm)$ and $\widetilde {s}(\Mm)$ are closely related to the fact that determining a regular oriented matroid is equivalent to finding a
{\em connected covering} of its elements by circuits. The latter generalizes the widely studied problem of {\em circuit covers} in matroids~\cite{Sey-80,Lem-06,mcG-10} and {\em cycle covers} in graphs~\cite{Fan-02,Lai-03,Yan-11}.
\medskip

$\bullet$ The study of $\widetilde{s}(U_{r,n})$ naturally leads us to consider the so-called {\em mutation} operation in uniform oriented matroids. A challenging question concerning mutations is the following famous simplex conjecture
 of Las Vergnas.
\begin{conjecture}
 Every uniform oriented matroid has at least one mutation.
\end{conjecture}

This conjecture is known to be true only for \emph{realizable} oriented matroids~\cite{Sha-79}, oriented matroids of rank at most $3$~\cite{Lev-26}, and for rank $4$ oriented matroids with few elements~\cite{Bok-01}. 

\subsection{Organization of the paper} In the next section, we recall some oriented matroid basics used throughout the paper.  

\medskip

In Section~\ref{sec:gen-bound} we give general bounds by introducing the notion of {\em weak} and {\em connected covering} (Theorem~\ref{thm:generalbounds}). We present an upper bound for $\overline{s}(\Mm)$ (Theorem~\ref{thm:generalboundssbar}) and also study the problem of determining an oriented matroid within its reorientation class (Theorem~\ref{thm:1reclass}).

\medskip

Section~\ref{sec:uniform} is devoted to the study of uniform oriented matroids. After recalling the relationship with Design Theory we recover an earlier result given in~\cite{For-98} in a more general framework (Theorem~\ref{prop:motiv}).
We provide exact values for $\widetilde{s}(U_{n-2,n})$ with $n\ge 3$ (Theorem~\ref{Th:unif}) as well as a general lower bound for $\widetilde{s}(U_{n-r,n})$ with $3\leq r\leq n-2$ (Theorem~\ref{Th:unif1}). We finally present the exact value of $\overline{s}(U_{r,n})$ with $1\leq r\leq n-1$ (Theorem~\ref{thm:s(n,r)})

\medskip

In Section~\ref{sec:regular} we turn our attention to regular matroids. We first notice that if $\Mm$ is regular then $\widetilde{s}(\Mm)=s(\Mm)$ and both equal the size of the smallest {\em connected element covering} of $\Mm$. We then give different bounds for the latter in the case when the matroid $\Mm$ is regular (Lemma \ref{lem:makeconnected}, Theorem \ref{thm:edgevertexbound}) and in particular when $\Mm$ is graphic (Proposition \ref{prop:makeconnectedbestpossible}, Corollary \ref{cor:2conn}).

\medskip

Finally, in Section~\ref{sec:graph} we calculate the values for the graphic and cographic matroids associated to complete graphs and hypercubes. More precisely, we calculate
$\widetilde{s}(\Mm)$ and $s(\Mm)$ when 
$\Mm$ is either $\Mm(K_n), \Mm^*(K_n), \Mm(Q_n)$ or $\Mm^*(Q_n)$ where $K_n$ is the complete graph on $n$ vertices and $Q_n$ is the hypercube graph of dimension $n$}
(Theorems~\ref{thm:Qn},~\ref{thm:Qn*},~\ref{thm:Kn} and~\ref{thm:Kn*}). These results show that some of our general bounds are tight.
 
\section{Basic definitions and general bounds}\label{sec:first}


Besides circuits we sometimes also use bases to represent a matroid. Given a basis $B$ of $\Mm$ and an element $e\notin B$, there is a unique circuit $C(B,e)$ of $\Mm$ contained in $B\cup\{e\}$ called the \emph{fundamental circuit} of $B$ with respect to $e$. In the oriented case two opposite orientations of this circuit appear. We denote by $C(B,e)$ any of them if no distinction is necessary.  So this means,
$C(B,e)$ can denote the (unoriented)
circuit of the matroid or either of the two corresponding circuits of the
oriented matroid.

A {\em basis orientation} of an oriented matroid $\M$ is a mapping $\chi_{\M}$ of the set of the ordered bases of $\Mm:=\underline{\mathcal{M}}$ to $\{-1,1\}$ satisfying the following properties :


\begin{itemize}
 \item[(CH1)] $\chi_{\M}$ is alternating.
 \item[(CH2)] for any two ordered bases $B, B'$ of $\Mm$ of the form $(e,b_2,\dots ,b_r)$ and $(e',b_2,\dots ,b_r)$, $e\neq e'$, we have  $\chi_{\M}(e,b_2,\dots ,b_r)=-C(B',e)_{e'}C(B',e)_e\chi_{\M}(e',b_2,\dots b_r)$, where $C(B',e)_e$ and $C(B',e)_{e'}$ denote the sign corresponding to elements $e$ and $e'$ in $C(B',e)$ respectively.
\end{itemize}

We have that $\mathcal{M}_1=\mathcal{M}_2$ if an only if $\chi_{\M_1}=\pm \chi_{\M_2}$.

We say that a base $B'$ of an oriented matroid $\M$ with chirotope $\chi$ is {\em invertible} if 

$$\chi^{B'}(B):=\left\{\begin{array}{rr}
-\chi_{\M}(B) & \hbox{if $B=B'$,}\\
\chi_{\M}(B) & \hbox{otherwise}\\
\end{array}\right.
$$

is also the chirotope of an oriented matroid $\M^{B'}$ (obtained thus from $\chi_{\M}$ by inverting only the sign of base $B'$). In the case of uniform oriented matroids invertible bases are called \emph{mutations}.

For every subset $A\subseteq E$ and every signed set $X$ of $E$, we denote by $_{\bar A}X$ the signed set obtained from $X$ by {\em reversing signs} on $A$, i.e.,  $(_{\bar A}X)^+=(X^+\setminus A)\cup (X^-\cap A)$ and
$(_{\bar A}X)^-=(X^-\setminus A)\cup (X^+\cap A)$. The set $\{_{\bar A}C \mid C\in\mathcal{C}\}$ is the set of signed circuits of an oriented matroid, denoted by $_{\bar A}\M$. Two oriented matroids $\M$ and $\M'$ are {\em related by sign-reversal} if $\M'=_{\bar A}\M$ for some $A\subseteq E$. The equivalence classes for this relation are called {\em reorientation classes}.  Notice that our definition of reorientation classes differs from the definition that is often used in the literature which applies to {\em unlabeled oriented matroids}, i.e., apart from sign-reversal also isomorphisms are allowed transformations. See the next subsection.

\subsection{Topological representation : quick discussion}

The well-known Topological Representation Theorem due to Folkman and Lawrence~\cite{Fol-78} states that loop-free oriented matroids of rank $d+1$ (up to isomorphism) are in one-to-one correspondence with arrangements of pseudospheres in $S^{d}$ (up to topological equivalence) or equivalently to affine arrangements of pseudohyperplanes in
$\mathbb{R}^{d-1}$ (up to topological equivalence).

\medskip

Note that, as mentioned above, in the literature contrary to our definition, the term reorientation class is often applied to \emph{unlabeled} oriented matroids. For instance the equivalence relation considered by the Topological Representation Theorem identifies two oriented matroids if they can be transformed via resignings, relabelings and reorientation into each other, see e.g. the book~\cite{Bjo-99}.
For example, $U_{2,n}$ has only one topological representation
but in our sense it admits exactly $\frac{(n-1)!}{2}$ reorientation classes~\cite{Cor-90}. Example~\ref{exmpl:U24} illustrates two of the three  reorientation classes of $U_{2,4}$. 

\section{General bounds}\label{sec:gen-bound}


We derive some necessary and sufficient conditions for a set  of circuits $\Ss$ to determine all the orientations of $\Mm$ or to determine a specific $\M$. In order to do this, we introduce some matroid parameters which {will be used} as upper and lower bounds. 

We say that a circuit $C$ of $\Mm$ \emph{covers} a basis $B$ if and only if there is an element $e\in E\setminus B$ such that $C$ is the fundamental circuit $C(B,e)$ of the basis $B$ with respect to $e$. A signed circuit $C$ of $\M$ \emph{covers} a basis $B$ of $\underline{\M}$ if  $\underline{C}$ covers $B$.

\begin{proposition}\label{prop:orientedweak} Let $\M$ be an oriented matroid. If $\S$ determines $\M$ then $\S$ covers all invertible bases of $\M$.
\end{proposition}

\begin{proof}
Let $B$ be an invertible basis  of $\M$ which is not covered by any signed circuit in $\S$. Let $\M'$ be the oriented matroid with chirotope $\chi^B$. By (CH2), the orientation of $\underline{\S}$ in $\mathcal{M}'$ depends only on signs of bases covered by $\underline{\S}$. Thus, the set $\S$ is a subset of the set of signed circuits of $\M'$.  Therefore, $\S$ does not determine $\M$.
\end{proof}

We say that a set $\S$ of signed circuits of $\M$ is a \emph{weak covering} of $\M$ if it covers all the invertible bases of $\M$. Let {\boldmath $\widetilde{\WC}(\Mm)$} be the smallest $k$ such that in each orientation $\M$ of $\Mm$ there is a weak covering of size $k$. Analogously, a set $\Ss$ of circuits of $\Mm$ is called a \emph{weak covering} of $\Mm$ if its orientation $\S_{\M}$ in any orientation $\M$ of $\Mm$ is a weak covering of $\M$. Let {\boldmath $\WC(\Mm)$} be the size of a smallest weak covering of $\Mm$. The following results are an immediate consequence of Proposition~\ref{prop:orientedweak}.

\begin{corollary}\label{cor:weak}
 For any orientable matroid $\Mm$ we have $\widetilde{\WC}(\Mm)\leq \widetilde{s}(\Mm)$ and $\WC(\Mm)\leq s(\Mm)$.
\end{corollary}

Given a set $\Ss$ of circuits of $\Mm$ define the graph $B_{\Ss}$ with vertex set the set of bases of $\Mm$ where $B, B'$ are adjacent if and only if $|B\Delta B'|=2$ and there is $C\in\Ss$ such that $C\subseteq B\cup B'$. 
Note that $B_{\Cc(M)}$ is just the base graph of $\Mm$.
A {\em base covering} is a set of circuits covering all the bases of $M$. A base covering $\Ss$ of $\Mm$ is called \emph{connected} if the graph $B_{\Ss}$ is connected.
Let {\boldmath $\CC(\Mm)$} be the size of a smallest connected base covering of $\Mm$.

\begin{theorem}\label{thm:generalbounds}
 For every connected orientable matroid $\Mm$ we have $s(\Mm)\leq \CC(\Mm)$.
\end{theorem}

\begin{proof}
Let $\Ss$ be a connected base covering of $\Mm$. Let $B$ be a basis and $B'$ a neighbor of $B$ in $B_{\Ss}$, i.e., $|B\Delta B'|=2$ and suppose there is $C\in\Ss$ such that $C\subseteq B\cup B'$. This means that there are $f\in B'$ and $e\in B$ such that $C=C(B',e)=C(B,f)$. By fixing the orientation of $\chi(B)$ using the signs of $C$ the orientation of $B'$ is determined via (CH2). 

Therefore, if $\Ss$ is a connected base covering the choice of
the value for $\chi(B)$ as well as the signings of the circuits in $\Ss$ induce a unique oriented matroid. Moreover, both choices for $\chi(B)=1$ or $=-1$ determine the same oriented matroids, with opposite chirotopes. 
\end{proof}

Let $\lambda(\Mm)$ be the largest $k$ such that for any set $\Ss$ of $k-1$ circuits the graph $B_{\Cc\setminus\Ss}$ is connected. We denote by $r(\Mm)$ the rank of $\Mm$.

\begin{theorem}\label{thm:generalboundssbar}
 For every connected orientable matroid $\Mm=(E,\Cc)$ we have $$\overline{s}(\Mm)\leq |\mathrm{C}|+1-\min(\lambda(\Mm), |E|-r(\Mm)).$$ If $\Mm$ has a base which is invertible in some orientation, then $|\mathrm{C}|+1-|E|+r(\Mm)\leq\overline{s}(\Mm)$.
\end{theorem}
\begin{proof}
 We start by proving the first inequality. If $\Ss\subseteq\mathrm{C}$ has size $|\mathrm{C}|+1-\min(\lambda(\Mm), |E|-r(\Mm))$, then $\Cc\setminus\Ss$ has cardinality $\min(\lambda(\Mm), |E|-r(\Mm))-1$. Therefore removing $\Cc\setminus\Ss$ cannot disconnect $B_{\Cc}$, i.e., $B_{\Ss}$ is connected. Removing $\Cc\setminus\Ss$ leaves no basis uncovered, since each base is covered by exactly $|E|-r(\Mm)$ circuits. Thus, $\Ss$ is a connected base covering of $\Mm$ and the result follows by Theorem~\ref{thm:generalbounds}.
 
 For the second bound let $B$ be an invertible basis of $\Mm$. Now choosing $\Ss$ as all circuits except those covering $B$ yields a set of size $|\mathrm{C}|-|E|+r(\Mm)$ which is not a weak covering of $\Mm$. The result follows by Corollary~\ref{cor:weak}.
\end{proof}

Indeed, we believe that the minimum in the upper bound in Theorem~\ref{thm:generalboundssbar} is always attained by $|E|-r(\Mm)$. We will see this for uniform oriented matroids in Theorem~\ref{thm:s(n,r)}. 

\subsection{Oriented matroids with one reorientation class}\label{subsec:reorie}

{In this subsection we give sufficient and necessary conditions for a set of circuits to determine a matroid \emph{within a given reorientation class}. More precisely, suppose that the oriented matroids, coinciding on a given set of circuits, lie in the same reorientation class. We shall study conditions yielding a unique orientation in this class.
Although our main results in this subsection (Theorems \ref{thm:1reclass} and \ref{thm:1reclasssbar}) are stated in terms of matroids having a single reorientation class they yield results for general matroids (Corollaries \ref{cor:generallb} and \ref{cor:generallbbar}).}
\smallskip

Given a matroid $\Mm$ an \emph{(element) covering} is a set $\Ss$ of circuits covering the ground set $E$. 
An element covering $\Ss$ is said to be \emph{connected} if the \emph{(element) intersection graph} $I_{\Ss}$ of $\Ss$ is connected. Let {\boldmath $\ccc(\Mm)$} be the size of a smallest connected element covering of $\Mm$.

\begin{theorem}\label{thm:1reclass}
 If a connected matroid $\Mm$ has a single reorientation-class, then we have $$\widetilde{s}(\Mm)=s(\Mm)=\ccc(\Mm).$$
 Moreover, these equalities are attained by the same fixed element covering $\Ss$ of $\Mm$.  
\end{theorem}
\begin{proof}
We first show $\ccc(\Mm)\geq s(\Mm)$. Let $\Ss$ cover $E$ and $I_{\Ss}$ be connected. Suppose there were two orientations $\M$ and $\M'$ of $\Mm$ coinciding on $\Ss$. By the preconditions $\M$ and $\M'$ differ by reorienting a set $X\subset E$. 
 
 We reorient $X$ in $\M$, but since all orientations of circuits of $\Ss$ shall be maintained, every circuit $C\in \Ss$ intersecting $X$ has to be reoriented entirely, i.e, $C\subseteq X$. Therefore all neighbors of $C$ in $I_{\Ss}$ are also contained in $X$. Iterating this argument \emph{all} circuits in $\Ss$ have to be completely reoriented. Since $\Ss$ covers $E$ all elements have to be reoriented, i.e., $X=E$. Thus, $\M=\M'$.
\medskip
 
We now show $\widetilde{s}(\Mm)\geq\ccc(\Mm)$. If $\Ss$ does not cover some $e\in E$, then in any orientation $\M$ of $\Mm$ we can reorient $e$ independently of the rest, i.e., $\M$ and the reorientation of $\M$ at $e$ coincide on $\Ss$. 
 If $I_{\Ss}$ has two connected components corresponding to two sets of circuits $\Ss', \Ss''$, then in any orientation $\M$ of $\Mm$ we can reorient all elements covered by $\Ss'$. Since all signs of signed circuits in $\S'$ are reversed, the resulting orientation $\M'$ coincides with $\M$ on $\Ss'$ and thus on $\Ss$. Nevertheless, reorienting $\S'$ in particular changes the orientation of circuits containing an element covered by $\Ss'$ and one covered by $\Ss''$. Therefore $\M'\neq \M$,
 
 Hence, if $\Ss$ is not a covering or $I_{\Ss}$ is disconnected then \emph{no} orientation $\M$ of $\Mm$ is determined by $\S_{\M}$.  The result follows by Observation~\ref{obs:three}.
\end{proof}

The following result is an immediate consequence of Theorem~\ref{thm:1reclass} and gives an alternative lower bound for $s(\Mm)$ to the one presented in Corollary~\ref{cor:weak}.

\begin{corollary}\label{cor:generallb}
 For any connected orientable matroid $\Mm$ we have $\ccc(\Mm)\leq \widetilde{s}(\Mm)$.
\end{corollary}

Theorem~\ref{thm:1reclass} allows us to say something about $\overline{s}$ for matroids with only one reorientation class. In order to prove the next result, we need  the following definition. For a matroid $\Mm$ with set of circuits $\Cc$, denote by $\kappa(I_{\Cc})$ the vertex connectivity of the graph $I_{\Cc}$. If $\Ss_e$ is the set of circuits containing a given element $e$ of the ground set $E$, we set  $\Delta(\Mm):=\max\{|\Ss_e| | e\in E\}$.

\begin{theorem}\label{thm:1reclasssbar}
 If $\Mm$ is a connected matroid with a single reorientation class, then
 $$ \overline{s}(\Mm)=|\Cc|+1-\min(\Delta(\Mm), \kappa(I_{\Cc})).
$$
\end{theorem}
\begin{proof}
For any set $\Ss$ of circuits of $\Mm$ with $|\Ss|>|\Cc|-\Delta(\Mm)$, we have that $\Ss$ is an element covering. Similarly, for any set $\Ss$ of circuits of $\Mm$ with $|\Ss|>|\Cc|-\kappa(I_{\Cc})$, the induced subgraph $I_{\Ss}$ of $I_{\Cc}$ is connected, otherwise there would be set $X\subset V(I_{\Ss})$ with $I_{\Ss}-X$ not connected and $|X|<\kappa(I_{\Cc})$ which is impossible. Thus, $\Ss$ is a connected element covering and with Theorem~\ref{thm:1reclass} we have $\overline{s}(\Mm)\le |\Cc|+1-\min(\Delta(\Mm), \kappa(I_{\Cc}))$.

\medskip

On the other hand, we note that there exists a set $\Ss$ of circuits of $\Mm$ with $|\Ss|\le |\Cc|-\Delta(\Mm)$ such that $\Ss$ is not an element covering. Similarly,  there exists a set $\Ss$ of circuits of $\Mm$ with $|\Ss|\le |\Cc|-\kappa(I_{\Cc})$ such that the induced subgraph $I_{\Ss}$ of $I_{\Cc}$ is not connected. Hence with Theorem~\ref{thm:1reclass} $\overline{s}(\Mm)\geq|\Cc|+1-\min(\Delta(\Mm), \kappa(I_{\Cc}))$.

\medskip

Together we get $\overline{s}(\Mm)=|\Cc|+1-\min(\Delta(\Mm), \kappa(I_{\Cc}))$.
\end{proof}

The following result is an immediate consequence of Theorem~\ref{thm:1reclasssbar} and complements the upper bound in Theorem~\ref{thm:generalboundssbar}.

\begin{corollary}\label{cor:generallbbar}
 For any connected orientable matroid $\Mm$ we have $|\Cc|+1-\min(\Delta(\Mm), \kappa(I_{\Cc}))\leq \overline{s}(\Mm)$.
\end{corollary}


\section{Uniform Oriented Matroids}\label{sec:uniform}

Let us quickly describe the connection of $s(U_{r,n})$ with connected coverings.
\smallskip

Let $n,k,r$ be positive integers such that $n\ge k\ge r\ge 1$. An \textit{$(n,k,r)$-covering} is a family $\mathrm{B}$ of $k$-subsets of $\{1,\ldots,n\}$, called \textit{blocks}, such that each $r$-subset of $\{1,\ldots,n\}$ is contained in at least one of the blocks. The number of blocks is the covering's \textit{size}. The minimum size of such a covering is called the \textit{covering number} and is denoted by {\boldmath $\C(n,k,r)$}. Given an $(n,k,r)$-covering $\mathrm{B}$, its graph $G(\mathrm{B})$ has $\mathrm{B}$ as vertices and two vertices are joined if they have one $r$-subset in common. We say that an $(n,k,r)$-covering is \textit{connected} if the graph $G(\mathrm{B})$ is connected. The minimum size of a connected $(n,k,r)$-covering is called the \textit{connected covering number} and is denoted by {\boldmath $\CC(n,k,r)$}.

\begin{theorem}[\cite{For-98,For-02}]\label{prop:motiv}
$$\C(n,r+1,r)\leq s(U_{r,n})\leq \CC(n,r+1,r).$$
\end{theorem}

In~\cite{For-02} a \emph{disconnected} covering determining all orientations of a uniform matroid is presented. However, its size is larger than the size of a smallest connected covering. 
\smallskip

We quickly recall some facts about oriented matroids needed in the rest of this section. Let $\M$ be a uniform oriented matroid and let $\mathcal{A}_{\M^*}$ be the pseudosphere arrangement representing the dual oriented matroid $\M^*$ of $\M$. The signed circuits $\mathcal{C}$ of $\M$ correspond to the cocircuits of $\M^*$ which are represented by the set of vertices ($0$-dimensional cells) of the arrangement $\mathcal{A}_{\M^*}$. A pair of oppositely signed circuits of $\M$ corresponds to an $S^0$ in $\mathcal{A}_{\M^*}$. Let $R_{B^*}$ be a full-dimensional simplicial cell in $\mathcal{A}_{\M^*}$ where $B^*$ is a base of $\M^*$ whose elements correspond to the bounding pseudospheres of $R_{B^*}$. We notice that any of the circuits corresponding to the vertices of $R_{B^*}$ in 
$\mathcal{A}_{\M^*}$ are circuits in $\M$ containing the base $B=E\setminus B^*$ because $\Mm$ is uniform. To see the latter, notice that the underlying set of each such circuit is formed by the pseudospheres not touching the corresponding vertex and so all the elements of $B$ will be included in such circuits. Finally, it is known that the mutations of $\M$ correspond to those bases corresponding to simplicial cells \cite{Rou-Sturm88}. Thus, in this section using Proposition~\ref{prop:orientedweak} we will encounter the problem of finding circuits touching all simplicial cells in an arrangement in order to obtain a weak covering of $\M$.  
\medskip

Let us give an alternative proof of Theorem \ref{prop:motiv} in a more general framework.

\begin{proof}[Proof of Theorem \ref{prop:motiv}]
We shall show that $\C(n,r+1,r)=\WC(U_{r,n})$ and $\CC(n,r+1,r)=\CC(U_{r,n})$. The claimed inequalities then follow by Theorem~\ref{thm:generalbounds} and Corollary~\ref{cor:weak}. 

First, note that the fundamental circuits of a base $B$ of $U_{r,n}$ are precisely the $(r+1)$-element sets containing $B$. Therefore the notions of $(n,r+1,r)$-covering and base covering of $U_{r,n}$ are the same.

For the first equality, it is enough to observe that for any base $B$ of $U_{r,n}$, there is an orientation with $B$ being invertible, i.e., a mutation, and so the result will follow by Proposition~\ref{prop:orientedweak}. 
So, let us take an $(n-r)$-simplex $R$ in $\mathbb{R}^{n-r}$. Define an affine hyperplane arrangement $\mathcal {A}_B$ consisting of the bounding hyperplanes of $R$ and $r$ further hyperplanes not intersecting $R$. We can label the bounding hyperplanes of $R$ with the elements of $B$. Since $R$ is a simplicial region of the arrangement $\mathcal {A}_B$, $B$ will be a mutation in any orientation of the hyperplanes of $\mathcal {A}_B$.

\smallskip

For the second equality we have to show that for a base covering $\Ss$ of $U_{r,n}$ we have that $G_{\mathrm{\Ss}}$ is connected if and only if $B_{\Ss}$ is connected. The crucial observation is that in $U_{r,n}$ a circuit $C$ covers $B$ and $B'$ if and only if $C=B\cup B'$. Therefore, there is a path from $C$ to $C'$ in $G_{\mathrm{\Ss}}$ if and only if there is a path from $B$ to $B'$ in $B_{\Ss}$ for all $B, B'$ covered by $C, C'$, respectively. Since $\Ss$ is a base covering, we obtain the result.

\end{proof}

%
\smallskip

Notice that by Observation~\ref{obs:three} we have $\widetilde{s}(\Mm)\le s(\Mm)$ and that Theorem~\ref{thm:1reclass} shows that both parameters are equal if the matroid has a single reorientation-class. It turns out that the inequality is strict for infinitely many matroids. Indeed, by~Theorem \ref{prop:motiv} and the fact that $\C(n,n-1,n-2)=\CC(n,n-1,n-2)=n-1$~\cite{Cha-13} we have that $s(U_{n-2,n})=n-1$ for every $n\ge 3$. {On the other hand, the following result shows that $\widetilde{s}(U_{n-2,n})$ is different from $s(U_{n-2,n})$ in general.}

\begin{theorem}\label{Th:unif} Let $n\ge 3$ be an integer. Then, 
$\widetilde{s}(U_{n-2,n})=\lceil\frac{n}{2}\rceil$.
\end{theorem}
\begin{proof}
We start by proving $\widetilde{s}(U_{n-2,n})\le \lceil\frac{n}{2}\rceil$. Let $\mathcal{M}$ be a uniform oriented matroid
of rank $n-2$ and let $\mathcal{A}$ be the topological representation of its dual. This
is, $\mathcal{A}$ is an arrangement of oriented pairs of antipodal points on a circle, i.e., several copies of $S^0$,
on an $S^1$ each dividing $S^1$ in a positive and a negative half. Each point corresponds to a signed circuit of $\mathcal{M}$. The complement
of each edge, i.e., complement of a closed segment of $S^1$ between two
consecutive points, corresponds to a basis of $\mathcal{M}$. We will consider the following set $\S$  of signed circuits of $\mathcal{M}$.  We choose points
from $\mathcal{A}$ to be part of $\S$ in an alternating way around $S^1$
starting at any point and continuing until $\S':=\S\cup -\S$ covers all edges. Clearly, $|\S|=\lceil\frac{n}{2}\rceil$. We prove that $\S$
determines $\mathcal{M}$, i.e., there is a unique arrangement $\mathcal{A}$ of
$n$ antipodal pairs yielding $\S$. 
Clearly, $\S$ gives that also $-\S$ are circuits. So, let us show that $\S':=\S\cup -\S$ determines $\mathcal{M}$. 
Take $\mathcal{A}'$ to be any arrangement having signed circuits $\S'$. First, observe that the subarrangement obtained by restricting to $\S'$ coincides with the restriction of $\mathcal{A}$ to $\S'$. (Both are representations of the same oriented matroid, corresponding to the restriction to the elements corresponding to $\S'$.)
Now, note that the signs
in $\S'$ determine the relative position of any point to points in $S'$. But
since $\S'$ covers all edges of $\mathcal{A}$ the relative position of a point
not in $\S'$ is between a unique pair of consecutive points of $\S'$ and no other
point is between them. Hence $\mathcal{A}'=\mathcal{A}$.
\smallskip

We now show that $\widetilde{s}(U_{n-2,n})\ge \lceil\frac{n}{2}\rceil$. We assume $|\S|<\lceil\frac{n}{2}\rceil$, then one edge of $\mathcal{A}$ is not incident to any element of $\S\cup -\S$. The oriented matroid arising by changing the order of the two copies of $S^0$ incident to that edge has different signs on the corresponding circuits, but does not differ on $\S$. This is a special case of Proposition~\ref{prop:orientedweak}.
\end{proof}

 
\begin{example}\label{exmpl:U24} Let  $\M_1$ and $\M_2$ be the orientations of $U_{2,4}$ which are the duals of the oriented matroids $\M'_1$ and $\M'_2$ induced by the topological representations given in Figure~\ref{fig:ex3}. 
 
 \begin{figure}[htb] 
 \includegraphics[width=.7\textwidth]{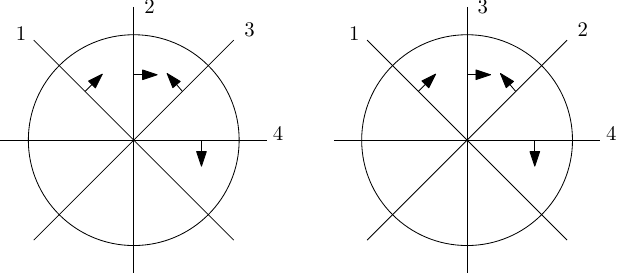}
 \caption{Two 1-dimensional oriented arrangements representing two orientations of $U_{2,4}$.} \label{fig:ex3}
 \end{figure}
 
 We clearly have that $\underline{\M_1}=\underline{\M_2}$ 
 since the sets of circuits of $\underline{\M_1}$ and $\underline{\M_2}$ coincide. However, $\M_1\neq\M_2$ since for their sets of signed circuits $\mathcal{C}_1,\mathcal{C}_2$ we have:
 $$\mathcal{C}_1=\{\{1,2,\bar{3}\},\{1,2,\bar{4}\},\{1,3,\bar{4}\},\{\bar{2},3,\bar{4}\}\}\neq \{\{1,\bar{2},3\},\{1,2,\bar{4}\},\{1,3,\bar{4}\},\{2,\bar{3},\bar{4}\}\}=\mathcal{C}_2.$$ 
 
Note that $1<\widetilde{s}(4,2)$. We may suppose that the circuit that had been chosen to determine $U_{2,4}$ was  $\S=\{1,2,\bar{4}\}$ which clearly does not determine $U_{2,4}$ since $\{1,2,\bar{4}\}$ is a signed circuit of of  $\M_1$ and $\M_2$.
 
Finally, it can be checked that there is no $A\subseteq \{1,2,3,4\}$ such that $_{\bar A}\M_1=\M_2$.
\end{example}


\begin{theorem}\label{Th:unif1}
 Let $3\leq r\leq n-2$. We have $(\frac{1}{2}(\lfloor \frac{n}{r-1}\rfloor+1))^{r-1}\leq\widetilde{s}(U_{n-r,n})$.
\end{theorem}
\begin{proof}
We define a simple affine pseudo-hyperplane arrangement in $\mathbb{R}^{r-1}$ in which almost every vertex is contained in exactly one simplex. Start with the \emph{grid}, i.e., the set of translates of coordinate hyperplanes $\mathcal{H}:=(H_i^k)_{i\in [r-1],k\in[\ell]}$ where $H_i^k:=\{x\in\mathbb{R}^{r-1}\mid x_i=k\}$. Now, we add the \emph{diagonal} hyperplanes $\mathcal{D}:=(D^j)_{r-1\leq j \leq k(r-1)}$ given by equations of the form $\sum_{i\in[r-1]}x_i=j+\epsilon$ for $r-1\leq j\leq k(r-1)$ and $0<\epsilon<1$. This, is $\mathcal{D}:=(D^j)_{r-1\leq j \leq k(r-1)}$ are the diagonals intersecting the grid translated a little bit into direction $(1, \ldots, 1)$. See Figure~\ref{fig:grid} for the rank $3$ case.

\begin{figure}[htb] 
 \includegraphics[width=.3\textwidth]{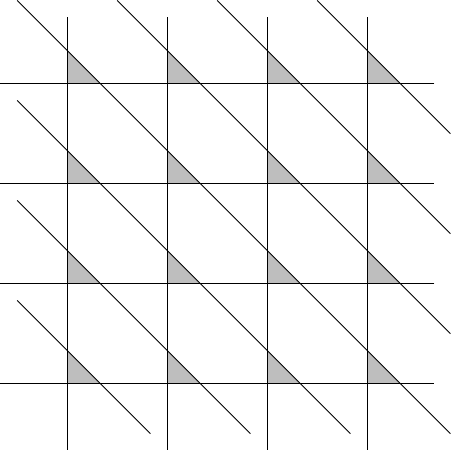}
\caption{The construction of Theorem~\ref{Th:unif1} for the case of rank $3$ and $\ell=4$. The gray cells are of type $R_v$.}
 \label{fig:grid}
\end{figure}

Note that in the resulting arrangement $\mathcal{H}\cup \mathcal{D}$ each vertex $v$ of $\mathcal{H}$ is incident to a unique simplex $R_v$ into direction $(1, \ldots, 1)$. Moreover, $R_v\cap R_w=\emptyset$ unless $v=w$.
Therefore, we need at least one vertex for each of these $\ell^{r-1}$ simplices.
We extend $\mathcal{H}\cup \mathcal{D}$ to an arrangement $\mathcal{A}$ representing an orientation of $U_{r,n}$ with $n=(\ell-1)(r-1)+1+(\ell)(r-1)=(2\ell-1)(r-1)+1$. Thus, to determine the dual of any oriented matroid arising from an orientation of $\mathcal{A}$ at least $(\frac{n+r-2}{2(r-1)})^{r-1}$ circuits are needed. So this is the lower bound if $n$ can be expressed as $(2\ell-1)(r-1)+1$. For general $n$ the argument of the lower bound is calculated by $(r-1)\lfloor \frac{n-1}{r-1}\rfloor$. A straight-forward computation leads to the claimed result.
\end{proof}

Even if we have shown that $s$ and $\widetilde{s}$ differ in general, one implication of the previous result is, that they are asymptotically the same for uniform oriented matroids.

A particular consequence of Theorem \ref{prop:motiv} shown in~\cite{Cha-13} is that $s(U_{r,n})$ behaves asymptotically as $\frac{1}{r+1} {n \choose r}$ for any fixed $r$.
We get that $s(U_{n-r,n})\in \Theta(n^{r-1})$. By combining Theorems~\ref{Th:unif} and~\ref{Th:unif1} and Observation~\ref{obs:three} we get
 
 \begin{corollary}\label{cor:stilde}
  For fixed $r\geq 1$ we have $\widetilde{s}(U_{n-r,n})=\Theta(n^{r-1})$.
 \end{corollary}

Let us now consider $\overline{s}$ for uniform oriented matroids.

\begin{theorem}\label{thm:s(n,r)}
 For any $1\leq r\leq n-1$ we have 
 $\overline{s}(U_{r,n})=\binom{n}{r+1}-n+r+1$.
\end{theorem}
\begin{proof}
Assume first that $n>r+1$. We will use Theorem~\ref{thm:generalboundssbar}. First, as argued in the proof of Theorem~\ref{prop:motiv} \emph{any} basis of $U_{r,n}$ is invertible in some orientation. Thus, we already have $$|\mathrm{C}|+1-|E|+r(M)\leq\overline{s}(\Mm)\leq |\mathrm{C}|+1-\min(\lambda(U_{r,n}), |E|-r(M)).$$ As we showed in the proof of Theorem~\ref{prop:motiv} for a set of $\Ss$ of circuits of $U_{r,n}$ we have that $B_{\Ss}$ is connected if and only if $G_{\Ss}$ is connected. Therefore $\lambda(U_{r,n})$ coincides with the vertex-connectivity $\kappa(G_{\Ss})$ of $G_{\Ss}$. But $G_{\Ss}$ is exactly the Johnson graph $J(n,r+1)$, see e.g.~\cite{God-01}. Now the vertex connectivity of $J(n,r+1)$ is well-known to be its degree, which is $(r+1)(n-r-1)$. On the other hand $|E|+r(M)=\binom{n}{r+1}-n+r$. We obtain the result.

If $n=r+1$ then $\binom{n}{r+1}-n+r+1=\binom{n}{r+1}$ is necessary and sufficient as well, because it indeed means taking all circuits and with one circuit less one could not cover one base.
\end{proof}

\section{Regular matroids}\label{sec:regular}

 In~\cite{Bla-78} it is shown that binary orientable matroids are exactly the regular matroids and that regular matroids have exactly one reorientation class. This section relies on these two facts. In particular, the first one leads us to give some results not depending on orientability when considering a general setting of binary matroids. The second one together with Theorem~\ref{thm:1reclass} immediately gives
 
\begin{corollary}\label{cor:regular}
 If $\Mm$ is regular and connected then $\widetilde{s}(\Mm)=s(\Mm)=\ccc(\Mm)$.
\end{corollary}

From now on we will focus on $\ccc(\Mm)$ rather than $s(\Mm)$. We will compare $\ccc(\Mm)$ with the minimum size of a (not necessarily connected) element covering of $\Mm$ denoted by {\boldmath $\cc(\Mm)$}. Moreover, since we will only consider element coverings rather than base coverings in this section and the section after, we will simply refer to them by coverings. In the present section we will derive several general bounds on $\ccc(M)$, which we will apply in the next section to some graphic and cographic matroids. 
\smallskip

Let us first reformulate the parameters $\cc(\Mm)$ and $\ccc(\Mm)$ when $\Mm$ is either a graphic or cographic. By Corollary~\ref{cor:regular} finding the circuits needed to determine all the orientations of a graphic matroid $\Mm(G)$ is equivalent to finding a set of cycles $\Ss$ in $G$ such that
\begin{itemize}
 \item every edge of $G$ is contained in some $C\in \Ss$,
 \item the graph induced by $\Ss$ (having as set of vertices the cycles of $\Ss$ and where two vertices $C$ and $C'$ are joined by an edge if and only if $C\cap C'\neq\emptyset$) is connected.
\end{itemize}

Such a set is called a \emph{connected cycle cover}. As for general matroids we denote the minimum size of such a set of cycles of $G$ by {\boldmath $\ccc(G)$}$=\ccc(\Mm(G))$. The size of a minimum (not necessarily connected) cycle cover is denoted by {\boldmath$ \cc(G)$}$=\cc(\Mm(G))$. 
\smallskip

A \emph{bond} $B$ in a connected graph $G$ is an edge-set which is inclusion-minimal with the property that $G\setminus B$ is disconnected. Finding the circuits needed to determine all the orientations of a cographic matroid $\Mm^*(G)$ is equivalent to finding a set of bonds $\Ss$ in $G$ such that
\begin{itemize}
 \item every edge of $G$ is contained in some $B\in \Ss$,
 \item the graph induced by $\Ss$ (having as set of vertices the bonds of $\Ss$ and where two vertices $B$ and $B'$ are joined by an edge if and only if $B\cap B'\neq\emptyset$)  is connected.
\end{itemize}

Such a set is called a \emph{connected bond cover}. We denote the minimum size of such a set of bonds of $G$ by {\boldmath $\cbc(G)$}. The size of a minimum (not necessarily connected) bond cover is denoted by {\boldmath $\bc(G)$}. This is, $\cbc(G)=\ccc(\Mm^*(G))$ and $\bc(G)=\cc(\Mm^*(G))$.
\smallskip

\begin{lemma}\label{lem:makeconnected}
 For any connected matroid $\Mm$ we have $\cc(\Mm)\leq \ccc(\Mm)\leq 2\cc(\Mm)-1$.
\end{lemma}
\begin{proof}
The first inequality is trivial and only stated for completeness. For the second one, 
 take a circuit cover $\Ss$ of $\Mm$ and let $C,C'$ be circuits not in the same component of $I_{\Ss}$. Since $\Mm$ is connected there is a circuit $C''$ incident to both $C$ and $C'$. Adding $C''$ to $\Ss$ reduces the number of components by at least one. This yields the claim.
\end{proof}

Indeed the upper bound in Lemma~\ref{lem:makeconnected} is best-possible as already shown by graphic matroids:
\begin{proposition}\label{prop:makeconnectedbestpossible}
 For every even $n\geq 2$ we have $\ccc(K_{2,n})=2\cc(K_{2,n})-1$.
\end{proposition}
\begin{proof}
 Clearly, in $K_{2,n}$ the longest cycles are of length $4$ and since $n$ is even a partition into $4$-cycles is possible. Thus, $\cc(K_{2,n})=\frac{n}{2}$. Now, given some set of cycles $\Ss$ in $K_{2,n}$, adding another cycle $C$ it can be incident to at most two components of $I_{\Ss}$. Thus, the construction in Lemma~\ref{lem:makeconnected} is best-possible.
\end{proof}

On the other hand there are cases, were $\cc$ and $\ccc$ coincide.

\begin{theorem}\label{thm:3}
 Let $\Mm$ be a binary and connected matroid. Denote by $\Cc^*_3$ the set of cocircuits of size at most $3$. If $\Cc^*_3$ covers $E$ and its intersection graph is connected, then any covering $\Ss$ of $\Mm$ is connected, i.e., $\cc(\Mm)=\ccc(\Mm)$.
\end{theorem}
\begin{proof}
 Let $\Ss$ be an element covering. Note that the existence of a circuit covering implies that there are no cocircuits of size $1$. Let $G_3$ be the connected intersection graph of $\Cc^*_3$. Now, every $C\in \Ss$ intersects elements of $\Cc^*_3$. Given $C,C'\in \Ss$ denote by $d_3(C,C')$ the length of a shortest path between two elements of $X,X'\in\Cc^*_3$ in $G_3$ such that $X$ intersects $C$ and $X'$ intersects $C'$. We claim that between every $C,C'\in \Ss$ there is a path in $I_{\Ss}$. We proceed by induction on $d_3(C,C')$.
 
 If $d_3(C,C')=0$, then there is $X\in\Cc^*_3$ intersecting both $C$ and $C'$. Since $\Mm$ is binary both $C$ and $C'$ intersect $X$ in an even number of elements. Since $|X|\leq 3$ both intersect $X$ in two elements and therefore $C\cap C'\neq\emptyset$. Thus, they are connected in $I_{\Ss}$.
 
 If $d_3(C,C')>0$, then choose a shortest path in $G_3$ witnessing $d_3(C,C')$. Let $X$ be the first cocircuit on this path, i.e., $X\cap C\neq \emptyset$.  As $\Mm$ is binary, We have that $|X\cap C|=2$. Since $X$ was the first member of a shortest path in $G_3$, there is an element $e$ with $X\setminus C=\{e\}$ and $e$ must be the intersection with the next member $X'$. Since $\Ss$ is an element covering, there is $C''\in \Ss$ containing $e$. As $\Mm$ is binary, we have that $C''$ intersects $C$. Thus, $d_3(C,C'')=0$ and $d_3(C'',C')<d_3(C,C')$. By induction hypothesis $C$ and $C''$ as well as $C''$ and $C'$ are connected in $I_{\Ss}$. This yields the claim.
\end{proof}

It is not sufficient to require that $E$ be covered by $\Cc^*_3$, see the left side of Figure~\ref{fig:notconnected}.
%
Also, the converse of Theorem~\ref{thm:3} does not hold as demonstrated by the right side of Figure~\ref{fig:notconnected}.

\begin{figure}[htb] 
 \includegraphics[width=.8\textwidth]{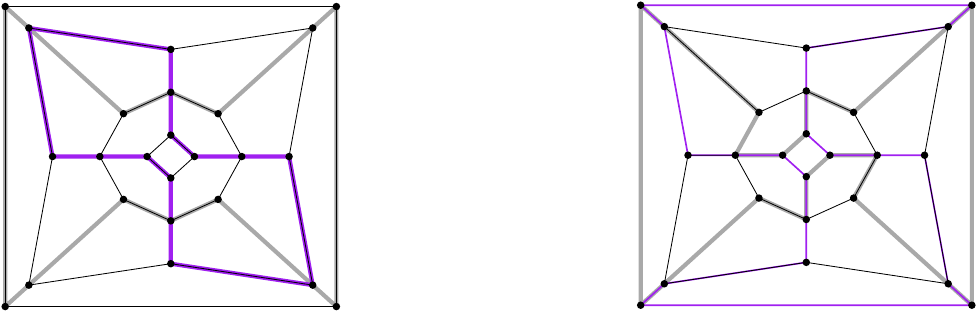}
 \caption{A graph $G$ in which every edge is contained in a bond of size $3$, but the intersection graph of these bonds is not connected. Left: A disconnected cycle covering consisting of four black, one purple and one gray cycle. Right: A connected cycle covering consisting of $3$ cycles. Indeed, one can see $\cc(G)=\ccc(G)=3$.}
 \label{fig:notconnected}
\end{figure}

Using Theorem~\ref{thm:3} and Lemma~\ref{lem:makeconnected} together with results on $\cc$ for $2$-connected graphs, ($\cc(G)\leq\lfloor\frac{2n-1}{3}\rfloor$, see~\cite{Fan-02}), and cubic graphs, ($\cc(G)\leq\lceil\frac{n}{4}\rceil$, if $G$ is cubic and $n\geq 6$, see~\cite{Lai-03} and $\cc(G)\leq\lceil\frac{n}{6}\rceil$, if $G$ is cubic, $3$-connected, $n\geq 8$, and $G$ is not one of five forbidden graphs, see~\cite{Yan-11}), we get some general bounds:
\begin{corollary}\label{cor:2conn}
 Let $G$ be a $2$-connected graph with $n$ vertices. Then,
  \begin{itemize}
                                                            \item $\ccc(G)\leq2\lfloor\frac{2n-1}{3}\rfloor-1$,
                                                            \item $\ccc(G)\leq\lceil\frac{n}{4}\rceil$, if $G$ is cubic and $n\geq 6$.
                                                            \item $\ccc(G)\leq \lceil\frac{n}{6}\rceil$, if $G$ is cubic, $3$-connected, $n\geq 8$, and $G$ is not one of five forbidden graphs.
                                                           \end{itemize}
\end{corollary}

We can also find some bounds involving the size of the ground set, the rank, the \emph{circumference} $\crc(\Mm)$, i.e., the size of the largest circuit of $\Mm$, and the \emph{cogirth} $g^*(\Mm)$, i.e., the size of a smallest cocircuit of $\Mm$.

\begin{theorem}\label{thm:edgevertexbound}
 For any regular matroid $\Mm$ we have $\frac{|E|-1}{\crc(\Mm)}\leq\ccc(\Mm)\leq |E|-r(\Mm)+2-g^*(\Mm)$.
\end{theorem}
\begin{proof}
 We start by proving the lower bound: The most optimistic way to find a connected covering is taking only circuits of maximal size, i.e., $\crc(\Mm)$. Moreover, since their intersection graph is connected they can be ordered such that each of them (except the first) shares at least one element with some earlier chosen one. Thus with $s$ such chosen circuits we cover $\crc(\Mm)+(s-1)(\crc(\Mm)-1)$ elements. So, this value should be at least $|E|$. From this we compute $s\geq \frac{|E|-1}{\crc(\Mm)}$.
 
 Given a matroid $\mathrm{\Mm}=(E,\mathrm{C})$, following~\cite{Lem-06} we denote by $\theta_e(\Mm)$ the size of a smallest set $\Ss'$ of circuits in $\Ss_e$ needed to cover $E$. Evidently, such $\Ss'$ is a connected element covering
 of $\Mm$ and thus $\ccc(\Mm)\leq \theta_e(\Mm)$ for all $e\in E$. Moreover, denote by $g^*_e(\Mm)$ the size of a smallest cocircuit containing $e$ and by $r(\Mm)$ the rank of $\Mm$. In~\cite[Corollary 1.5]{Lem-06} it is shown that if $\Mm$ is connected, regular and not a coloop, and $e\in E$ such that $\Mm/e$ is connected, then $\theta_e(\Mm)+g^*_e(\Mm)\leq |E|-r(\Mm)+2$. This immediately gives the result.
 
%
 \end{proof}

%

 A binary matroid is called \emph{Eulerian} if all cocircuits are of even size.
 \begin{lemma}\label{lem:Eulerian}
  If $\Mm$ is an Eulerian matroid, then $\ccc(\Mm)\neq 2$.
 \end{lemma}
 \begin{proof}
  Suppose $\ccc(\Mm)=2$ witnessed by circuits $C_1, C_2$ covering the entire ground set and $C_1\cap C_2=N\neq\emptyset$. Every cocircuit $X$ is even and since $\Mm$ is binary $X$ intersects both $C_1$ and $C_2$ in an even number of elements. This implies, that $|X\cap N|$ is even for all cocircuits $X$. It is a well-known fact that if $N\neq\emptyset$ and, for each cocircuit $X$, $|X\cap N|\neq 1$, then $N$ contains a circuit. This contradicts $C_1$ and $C_2$ being circuits.
 \end{proof}

Even if Lemma~\ref{lem:Eulerian} seems relatively weak, it provides tight lower bounds in a large family of cographic matroids as we will see in the next section.

\section{The hypercube and the complete graph}\label{sec:graph}

In this section we determine $\cc, \ccc, \bc$, and $\cbc$ for the class of hypercubes and complete graphs, where $Q_n$ is the \emph{$n$-dimensional hypercube} consisting of vertices $\{0,1\}^n$ connected by an edge whenever they differ in exactly one coordinate.  We will make use of some lemmas of the previous section and prove some bounds to be tight.
The next result for odd $n$ shows that the lower bounds in Theorem~\ref{thm:edgevertexbound} and Lemma~\ref{lem:makeconnected} can indeed be attained:

\begin{theorem}\label{thm:Qn}
 For every $n\geq 3$ we have  $$
\ccc(Q_n)=\left\lceil\frac{n+1}{2}\right\rceil= \begin{cases}
 \cc(Q_n)+1 & \text{if }n\text{ even}, \\
 \cc(Q_n) & \text{if }n\text{ odd}.\\
\end{cases}
$$
\end{theorem}
\begin{proof} 
 For the upper bound if $n$ is even we use that by~\cite{Als-90} the edges of $Q_n$ can be partitioned into $\frac{n}{2}$ Hamiltonian cycles, which proves $\cc(Q_n)=\frac{n}{2}$ in this case. Clearly, since this is a partition into Hamiltonian cycles no connected covering with $\frac{n}{2}$ cycles exists. Now take a bond $X$ corresponding to a change in one coordinate of $Q_n$. Since $X$ is a bond all Hamiltonian cycles have to intersect it. Now, since the coordinate corresponding to $X$ can be switched at any vertex of $Q_n$, $X$ is a perfect matching. Therefore $X$ can be extended to a Hamiltonian cycle, which intersects all the others, see~\cite{Fin-07}. This concludes the case $n$ even.
 \medskip
 
 If $n$ is odd, first note that $\ccc(Q_n)\geq\lceil\frac{n+1}{2}\rceil$ results from just plugging the values into Theorem~\ref{thm:edgevertexbound}. 
 For the upper bound choose two copies of $Q_{n-1}$, where vertices $v,v'$ in different copies correspond to each other in the natural way, i.e., $v$ corresponds to $v'$ if $v=(x,0)$ and $v'=(x,1)$. Denote by $X_n$ the matching induced between corresponding vertices in $Q_n$. Now, take the partition $P$ of $Q_{n-1}$ into Hamiltonian cycles and its copy $P'$ partitioning the copy $Q'_{n-1}$. The coordinate matching $X_{n-1}$ in $Q_{n-1}$ intersects every Hamiltonian cycle in $P$. Denote by $X\subseteq X_{n-1}$ a matching hitting each cycle in $P$ exactly once and by $X'$ its copy in $Q'_{n-1}$.
 For every Hamiltonian cycle $C\in P$ and its copy $C'\in P'$ take the unique matching edges of $e=\{v,w\}\in X$ and $e'=\{v',w'\}\in X'$ intersecting precisely $C$ and $C'$, respectively. Delete them from $C$ and $C'$ and join both cycles by adding $\{v,v'\}$ and $\{w,w'\}$, to obtain a new cycle $C''$. We have obtained a set $\widetilde{\Ss}$ of $\frac{n-1}{2}$ cycles. The edges of $Q_n$ still not covered by $\widetilde{\Ss}$ are precisely $X\cup X'$ and all edges of $\{u,u'\}\in X_n$ with $u$ not incident to $X$. Note that this set of edges forms a perfect matching $\widetilde{X}$. We have to cover $\widetilde{X}$ by a Hamiltonian cycle $\widetilde{C}$ which additionally intersects all cycles in $\widetilde{\Ss}$. For the construction we contract $X_n$ and $X_{n-1}\cup X'_{n-1}$ in $Q_n$ obtaining $Q_{n-2}$ (with parallel edges). Every vertex in $Q_{n-2}$ corresponds to a square in $Q_n$, which contains either two edges of $X\cup X'$ or two edges of $X_n$, which still have to be covered respectively. We call a vertex of type $I$ and $II$, depending on this. Moreover, by construction both remaining edges in such a square belong to exactly one $C\in\widetilde{\Ss}$. We say that a vertex of $Q_{n-2}$ meets $C$. Note that there are $\frac{n-1}{2}$ vertices of type $I$ each meeting one of the $\frac{n-1}{2}$ cycles in $C\in\widetilde{\Ss}$. All remaining vertices are of type $II$. 
 
 If $\frac{n-1}{2}$ is even let $H$ be any Hamiltonian cycle in $Q_{n-2}$. We can blow $H$ up to the desired $\widetilde{C}$ in $Q_n$ by just locally prescribing how to behave in the resulting squares. See the left of Figure~\ref{fig:cube}. Filling a type $I$ square corresponds to changing the coordinate inside the square along $X_n$ and a type $I$ square along $X_{n-1}\cup X'_{n-1}$. By the parity assumption and since $Q_{n-2}$ has an even number of vertices the numbers of type $I$ and type $II$ vertices are even. And our construction closes nicely and gives a cycle in $Q_n$.
 
 \begin{figure}[htb] 
 \includegraphics[width=.5\textwidth]{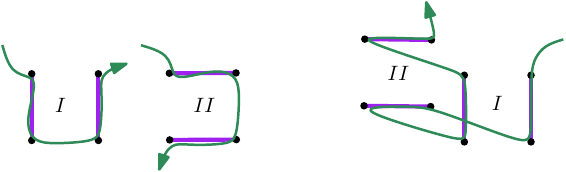}
 \caption{Blowing up $H$. Left: usual handling of vertices of type $I$ and $II$. Right: special treatment of the edge $e$. Purple edges are those needed to be covered by $\widetilde{C}$.}
 \label{fig:cube}
\end{figure}
 
 If $\frac{n-1}{2}$ is odd chose an edge $e$ in $Q_{n-2}$ connecting a type $I$ vertex $a$ meeting $C$ and a type $II$ vertex $b$, such that some other vertex in $Q_{n-2}$ different from $a,b$ meets $C$. Such an edge clearly exists; since otherwise all neighbors of type $II$ of $a$ different from $b$ meet a cycle different from $C$ and we can as well take an edge connecting $a$ with such a neighbor. Now, we choose a Hamiltonian cycle $H$ containing $e$. In order to obtain $\widetilde{C}$ we handle all vertices as in the case before except $a,b$, see the right of Figure~\ref{fig:cube}. Along $e$ we have to repair the parity in order to close to a cycle in $Q_n$. The choice of $e$ was complicated by the fact that $\widetilde{C}$ will not meet the cycle met by $a$ in the square resulting from $a$.
 \end{proof}

\begin{theorem}\label{thm:Qn*}
For every $n\ge 2$ we have  $\cbc(Q_n)=\bc(Q_n)+1=3$.
\end{theorem}
\begin{proof}
Let us prove first by induction on $n$ that $\bc(Q_n)=2$ for every $n\ge 2$. Clearly the proposition holds if $n=2$ and maybe a little less clearly also for $n=3$, see the left of Figure~\ref{fig:Q_n}.  We suppose the result holds for $n-1$ and $n>2$. We can obtain the graph $Q_{n}$ as follows. Choose two copies of $Q_{n-1}$, $A$ and $B$, where vertices in different copies correspond to each other in the natural way, i.e., adding a matching connecting identified pairs yields $Q_{n}$. By induction, there exists a partition of the edges of $A$ in two bonds $A_1$ and $A_2$. Let $B_i$ be the copy of $A_i$ in $B$ for $i=1,2$.
\medskip
Let us define as $[C,C']$, the sets of edges in a graph $G$ having one extreme in $C$ and the other in $C'$, for every $C,C'\subset V(G)$. Let $C_1^{A_j}$ and $C_2^{A_j}$ be the two components of $A-A_j$,  for $j=1,2$. 
Since $C_i^{A_j}$ has a copy in $B$, let $C_i^{B_j}$ be the copy of $C_i^{A_j}$ in $B$ for every $i,j=1,2$. Observe that $C_i^{B_j}$ is one component in $B-B_j$. We consider the following sets of edges:

\begin{center}
\begin{tabular}{ccc}
$E_1$ & $=$ & $E(C_1^{A_1})\cup [C_1^{A_1}\cap C_1^{A_2}, C_1^{B_1}\cap C_1^{B_2}]\cup E(C_1^{B_2})$\\
$E_2$ & $=$ & $E(C_2^{A_1})\cup [C_2^{A_1}\cap C_2^{A_2}, C_2^{B_1}\cap C_2^{B_2}]\cup E(C_2^{B_2})$\\
$E_3$ & $=$ & $E(C_2^{A_2})\cup [C_1^{A_1}\cap C_2^{A_2}, C_1^{B_1}\cap C_2^{B_2}]\cup E(C_1^{B_1})$\\
$E_4$ & $=$ & $E(C_1^{A_2})\cup [C_2^{A_1}\cap C_1^{A_2}, C_2^{B_1}\cap C_1^{B_2}]\cup E(C_2^{B_1})$\\
 
\end{tabular} 
\end{center}

Notice that $(E_1\cup E_2)\bigcup(E_3\cup E_4)=E(Q_{n})$ and $(E_1\cup E_2)\bigcap (E_3\cup E_4)=\emptyset$ which means that $E_1\cup E_2$ and $E_3\cup E_4$ are a partition of the edges of $Q_{n}$. We will see that $E_1\cup E_2$ and $E_3\cup E_4$ are two bonds of $Q_{n}$. For, we can check that $C_1^{A_2}\cup  C_2^{B_1}$  and  $C_2^{A_2}\cup  C_1^{B_1}$ are the two components of $Q_{n}-(E_1\cup E_2)$. Also, we can observe that each edge in $E_1\cup E_2$ is incident to a vertex in $C_1^{A_2}\cup  C_2^{B_1}$  and incident to a vertex in $C_2^{A_2}\cup  C_1^{B_1}$, which means that $E_1\cup E_2$ is a  bond of $Q_{n}$, see Figure~\ref{fig:Q_n}. Similarly, one can see that $E_3\cup E_4$ is a bond of $Q_{n}$.
Then $\bc(Q_n)=2$ for every $n\ge 2$.

\begin{figure}[htb] 
 \includegraphics[width=.9\textwidth]{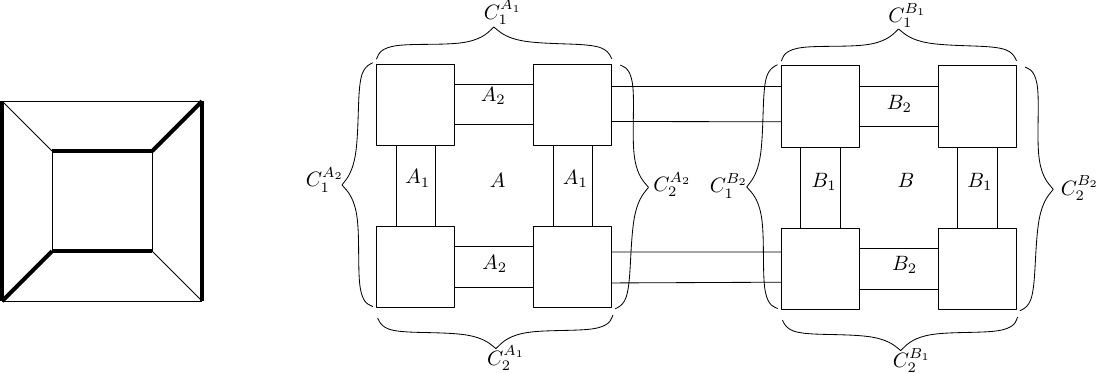}
\caption{Left: a partition of $Q_3$ into two bonds. Right: Extending a partition of $Q_{n-1}$ to a partition of $Q_n$.}
 \label{fig:Q_n}
\end{figure}

Let us prove now that $\cbc(Q_n)=3$ for every $n\ge 2$. We have proved above that there exist a partition of the edges of $Q_n$ in two bonds $A_1$ and $A_2$. Let $C$ be a bond in $Q_n$ different to $A_1$ and $A_2$. Then 
$E(A_1)\cap E(C)\neq\emptyset$ and $E(A_2)\cap E(C)\neq\emptyset$ and hence $\cbc(Q_n)\le 3$ for every $n\ge 2$.
To prove  that $\cbc(Q_n)\ge 3$, we only have to use Lemma~\ref{lem:Eulerian}. As the cographic matroid of $Q_n$ is Eulerian  since $Q_n$ is bipartite,
by Lemma~\ref{lem:Eulerian} it follows that  $\cbc(Q_n)\ge 3$ for every $n\ge 2$.

\end{proof}

%
%
%
%
%
%
 
For the complete graph $K_n$ the lower bounds of Theorem~\ref{thm:edgevertexbound} and Lemma~\ref{lem:makeconnected} are sharp if and only if $n$ is even. More precisely:
\begin{theorem}\label{thm:Kn}
 For every $n\geq 4$ we have 
 $$
\ccc(K_n)=\lceil\frac{n}{2}\rceil= \begin{cases}
 \cc(K_n) & \text{if }n\text{ even}, \\
 \cc(K_n)+1 & \text{if }n\text{ odd}.\\
\end{cases}
$$

\end{theorem}
\begin{proof}
We will view $K_n$ as the Cayley graph of $\mathbb{Z}_n$ with connecting set $\mathbb{Z}_n\setminus\{0\}$, where the double edges oriented into opposite directions are seen as undirected edges.

Our construction is based on a well-known partition $\mathcal{P}_n$ of the edges of $K_n$ for even $n$ into $\frac{n}{2}$ Hamiltonian zig-zag-paths, see e.g. \cite{Alsp-Gav01}. The path $P_i$ traverses the vertices as follows: $i, i+1,i+1-2,i+1-2+3,\dots$.
\medskip

Now, if $n$ is odd, take $\mathcal{P}_{n-1}$ and add an edge from the additional vertex $n$ to $P_i$ if it connects to one of the endpoints of $P_i$. This is a well-known construction for a partition of $K_n$ into Hamiltonian cycles. Thus, to obtain a connected cycle covering at least one additional cycle is needed. Indeed, such a cycle is easy to find. Take for instance $C=(0,1,\ldots, \frac{n+1}{2},0)$. By construction of $\mathcal{P}_{n-1}$ this cycle intersects all other 
cycles in the partition.
\medskip

For $n$ even it is well-known that $\cc(K_n)=\frac{n}{2}$, thus this lower bound on $\ccc(K_n)$ follows from Lemma~\ref{lem:makeconnected}. We will show that there are indeed connected cycle covers of that size.

Now, if $n$ is even but not divisible by $4$, take $\mathcal{P}_{n}$ and add an edge to each $P_i$ connecting its end-vertices. This yields a cover $\mathcal{H}_{n}$ of $K_n$ by Hamiltonian cycles $H_i$, which is smallest possible. We show that it is indeed connected. Note that each $H_i$ contains two \emph{long diagonals}, i.e., edges labeled $\frac{n}{2}$, one being the $\frac{n}{2}$th or middle edge $e_i$ of $P_i$ and one being the newly added edge $f_i$. Each such edge is contained in another element of $\mathcal{H}_{n}$. More precisely, we have $e_i=f_{i+\frac{n+2}{4}}$ and therefore in the graph on $\mathcal{H}_{n}$ there is an edge between $H_i$ and $H_{i+\frac{n+2}{4}}$, but $i+\frac{n+2}{4}$ has to be taken modulo $\frac{n}{2}$. By the divisibility conditions on $n$ we get that $\frac{n}{2}$ and $\frac{n+2}{4}$ are coprime and therefore connecting $H_i$ and $H_{i+\frac{n+2}{4}}$ modulo $\frac{n}{2}$ for all $i$ yields a single connected component. This is, $\mathcal{H}_{n}$ is a connected cycle cover of size $\frac{n}{2}$.

The last case concerns $n$ divisible by $4$. If $n=4$ it is easy to find a connected cycle cover of size $2$. Otherwise we take the cycle cover constructed in the paragraph above for $K_{n-2}$ and modify it to cover the complete graph with two additional vertices $v,w$. In $K_{n-2}$, each long diagonal is covered twice. In each cycle $H\in\mathcal{H}_{n-2}$ replace the long diagonal by two consecutive edges passing through $v$ and $w$, respectively.
Denote the resulting set of cycles by $\mathcal{H}'$. It covers all edges but $\{v,w\}$ and the long diagonals of $K_{n-2}$, i.e., edges labeled $\frac{n-2}{2}$ connecting vertices different from $v,w$. We add one more cycle $C$ using all these edges and taking every other edge of the cycle $(0,1,2,\ldots, n-3,0)$ of $K_{n-2}$ except $\{0,1\}$ and $\{\frac{n-2}{2},\frac{n-2}{2}+1\}$. Instead $C$ includes $\{\frac{n-2}{2},v\}$ and $\{\frac{n-2}{2}+1,w\}$ (or $v$ and $w$ permuted) such that on these edges $C$ intersects the cycle arising from $H_0$. All the other $H_i$ are intersected by $C$ via every other edge of the cycle $(0,1,2,\ldots, n-3,0)$ of $K_{n-2}$ since $n-2$ is not divisible by $4$.
\end{proof}

\begin{theorem}\label{thm:Kn*}
 For all $n\geq 2$ we have $\cbc(K_n)=\bc(K_n)=\lceil\log_2(n)\rceil$.
\end{theorem}
\begin{proof}
 First it is easy to see that $\Mm^*(K_n)$ satisfies the preconditions of Theorem~\ref{thm:3}, i.e., every edge is contained in a triangle and the edge-intersection graph of triangles is connected. This implies $\cbc(K_n)=\bc(K_n)$. Now, any set of minimal cuts covering the edges of $K_n$ corresponds to a set of maximal bipartite subgraphs covering the edges. Note that this correspondence holds if and only if the graph is the complete graph. The minimum number of bipartite subgraphs to cover a graph $G$ is $\lceil\log_2(\chi(G))\rceil$, see~\cite{Mat-72,Had-75,Gar-76}. In our special case it yields the result.
\end{proof}

\bibliography{circlit}
\bibliographystyle{my-siam}

\end{document}